\documentclass{amsart}
\newtheorem{theorem}{Theorem}[section]
\newtheorem{lemma}[theorem]{Lemma}
\newtheorem{corollary}[theorem]{Corollary}
\newtheorem{proposition}[theorem]{Proposition}

\theoremstyle{definition}
\newtheorem{definition}[theorem]{Definition}

\theoremstyle{remark}
\newtheorem{remark}[theorem]{Remark}

\numberwithin{equation}{section}

\def \a{{\alpha}}
\def \b{{\beta}}
\def \D{{\Delta}}

\def \d{{\delta}}
\def \e{{\varepsilon}}

\def \g{{\gamma}}

\def \k{{\kappa}}
\def \l{{\lambda}}

\def \o{{\omega}}

\def \p{{\varphi}}
\def \t{{\vartheta}}

\def \m{{\mu}}
\def \s{{\sigma}}

\def \N{{\bf N}}

\def \qq{{\qquad}}
\def \R{{\bf R}}

\def \T{{\bf T}}

\def \ua{{\bf a}}
\def \uc{{\bf c}}

\def \Z{{\bf Z}}

\def \dd{{\rm d}}

\def \noi{{\noindent}}

\def \T{{\mathbb T}}

\def\R{{\mathbb R}}

\def\Z{{\mathbb Z}}

\def\N{{\mathbb N}}

at  8 pt
\scrollmode

\font\gsec= cmb10 at 10 pt
   at 10,4  pt at  8,2 pt
  \scrollmode
\font\gsec= cmb10 at 9,8  pt
 at 8,2  pt

\begin{document}

\title{On  series of dilated functions}

 \author{
 Istv\'an Berkes and
Michel   WEBER}
 \address{Michel Weber: IRMA, 10  rue du G\'en\'eral Zimmer, 67084
 Strasbourg Cedex, France}
  \email{michel.weber@math.unistra.fr \!;  \! m.j.g.weber@gmail.com}
\address{Istvan Berkes, Technische Universitat Graz, Institut
f\"ur Statistik, M\"unzgrabenstrasse  11, A-8010 Graz, Austria.}
\email{berkes@tugraz.at}


\keywords{dilated functions, mean convergence, almost convergence,
orthogonal  systems, GCD matrix, eigenvalues, quadratic forms}

\begin{abstract}
Given a periodic function $f$, we study the almost everywhere and
norm convergence of series  $\sum_{k=1}^\infty c_k f(kx)$. As the
classical theory shows, the behavior of such series is determined by
a combination of analytic and number theoretic factors, but precise
results exist only in a few special cases. In this paper we use
connections with orthogonal function theory and GCD sums to prove
several new results, improve old ones and also to simplify and unify
the theory.
\end{abstract}

\maketitle

\section{\bf Introduction and preliminary results}

\bigskip
Let $\T=\R/\Z\simeq[0,1[$, $e(x)=\exp(2i\pi x)$, $e_n(x)= e(nx)$,
$n\in {\Z}$. Let $ f(x)= \sum_{\ell\in {\Z}} a_\ell e_{\ell }   $,
$a_0=0$, $\sum_{\ell \in {\Z}} |a_\ell|^2<\infty$. The convergence
and asymptotic properties of sums
\begin{equation}\label{fsum}
\sum_{k=1}^\infty c_k f(kx)
\end{equation}
have been studied extensively in the literature and turned out to
be, in general, quite different from the trigonometric case
$f(x)=e(x)$. (For history and recent results,  see e.g.\ Berkes and
Weber \cite{BW}.) For $f(x)=e(x)$ (and consequently, if $f$ is a
trigonometric polynomial), Carleson's theorem states the almost
everywhere convergence of (\ref{fsum}) provided ${\bf c}=\{c_k, k\ge
0\}\in \ell^2(\N)$. Using a simple approximation argument, Gaposhkin
\cite{G} showed that this remains valid if the Fourier series of $f$
is absolutely convergent, i.e.\ if $\sum_{\ell\in {\mathbb Z}}
|a_\ell|<\infty$. In particular, this is the case if $f$ belongs to
the $\text{Lip}_\alpha(\T)$ class for some $\alpha>1/2$. For
$\sum_{\ell\in {\mathbb Z}} |a_\ell|=\infty$ the convergence
properties of the sum (\ref{fsum}) are much more complicated and
satisfactory results exist  only in special situations. If the
Fourier series of $f$ is lacunary, i.e.\ if
$f(x)=\sum_{\ell=1}^\infty a_\ell e(n_\ell x)$, where
$n_{\ell+1}/n_\ell \ge q>1$, $\ell=1, 2, \ldots$, then by a theorem
of \cite{G}, the analogue of Carleson's theorem holds for
(\ref{fsum}) provided the $L_2$ modulus of continuity $\omega_2 (f,
h)$ of $f$ satisfies
$$\sum_{k=1}^\infty \frac{\omega_2 (f, 2^{-k})}{\sqrt{k}}<\infty$$
and this result is sharp. For general $f$, this criterion becomes
false:  if the Fourier series $f=\sum_p a_pe(px)$ ($\sum_p
|a_p|=\infty$) contains only prime frequencies, then the analogue of
Carleson's theorem is  false, even though this class contains
Lip\,(1/2) functions $f$, see Berkes \cite{Be}. It is also known
that the asymptotic distribution of sums $\sum_{k=1}^N c_k f(n_kx)$
depends sensitively on the Diophantine properties of the sequence
$(n_k)$, see e.g.\ Gaposhkin \cite{g70}. These results, together
with Wintner's classical criterion \cite{Wi} connecting the behavior
of (\ref{fsum}) with boundedness properties of the Dirichet series
$\sum_{n=1}^\infty a_n/n^s$,  show that the convergence properties
of (\ref{fsum}) are determined by an interplay of analytic and
number theoretic factors. Recent results of Weber \cite{W1} and
Br\'emont \cite{Br} shed a  new light on the arithmetic background of
the problem and led to much improved convergence results. Weber
showed that assuming a condition for $f$ only slightly stronger than
$f\in L^2$, the series (\ref{fsum}) converges a.e.\ provided
\begin{equation*}
\sum_{r=1}^\infty \bigg(\sum_{j=2^r+1}^{2^{r+1}} c_j^2 d(j) (\log
j)^2\bigg)^{1/2}<\infty,
\end{equation*}
where $d(k)=\sum _{d|k} 1$ is the divisor function. Br\'emont  showed
that under $|a_n|= O(n^{-s})$, $1/2<s\le 1$, a.e.\ convergence holds
if
that  the series $\sum_{k }  c_k f_k$ converges almost everywhere if for some $\e>0$
\begin{equation}\label{conds0}   \sum_{k }    c_k^2   \exp\Big\{\frac{(1+\e)(\log k)^{2(1-s)}}{2(1-s)\log\log k}\Big\} <\infty,
  \end{equation}
when $1/2<s<1$, and if for some $\e>0$
   \begin{equation}\label{conds1}   \sum_{k }    c_k^2( \log k)^3(\log\log k)^{ 2+\e}  <\infty,
  \end{equation}
when $s=1$. 
\vskip 2pt
 The purpose of the present
paper to give a detailed  study of the series (\ref{fsum}), using
connections with orthogonal function theory and asymptotic estimates
for GCD sums in Diophantine approximation theory. This will not only
lead to an extension and improvement of  earlier results, but will
also simplify and unify the convergence theory.

\medskip
The convergence behavior of (\ref{fsum}) is naturally closely
connected with estimating the quadratic form
\begin{equation}\label{norm}
\big\|\sum_{k=1}^n c_k f_k\big\|_2^2= \sum_{k,\ell =1}^n c_kc_\ell
\langle f_k, f_\ell\rangle .
\end{equation}
 Let us first study
it   on a simple class of  examples. We follow \cite{LS}. Consider the function
\begin{equation}\label{funct} f(x)= \sum_{j=1}^\infty \frac{\sin 2\pi jx}{j^s},
\end{equation}
where $s>1/2$. When $s=1$,  $f(x)= x-[x] -1/2$, where  $[x]$ is the
integer part of $x$. It is known (see \cite{J}) that the
corresponding system $\{f_n, n\ge 1\}$ possesses properties going at
the opposite of those of the trigonometrical system.
 \vskip2 pt
 A simple calculation yields
\begin{eqnarray} \label{sp}\langle f_k, f_\ell\rangle= \sum_{i,j=1\atop ik =j\ell }^\infty \frac{1}{i^sj
^s}=\Big(\sum_{\nu=1}^\infty\frac{1}{\nu^{2s} }\Big)\frac{(k,\ell)^{2s}}{k^s\ell^s}=\zeta(2s)\frac{(k,\ell)^{2s}}{k^s\ell^s},
\end{eqnarray}
where $\zeta$ is  Riemann's zeta function, and $(a,b)$ denotes the
greatest common divisor of the positive integers $a$ and $b$. It
follows that
\begin{equation*} 
\big\|\sum_{k=1}^n c_k f_k\big\|_2^2=\zeta(2s)\sum_{k,\ell=1}^n \frac{(k,\ell)^{2s}}{k^s\ell^s}c_k   c_\ell.
\end{equation*}
Subsequently, the GCD matrix
$$M_n(s)=\Big(\frac{(k,\ell)^{2s}}{k^s\ell^s}\Big)_{n\times n} $$
is positive definite when $s>1/2$. The study of this important class of matrices is much  older, and goes back to Smith's seminal
paper published in 1876 (see \cite{HSW} and references therein). Let
$\l_n(s)$ (resp.
$\Lambda_n(s)$) denote the smallest (resp. largest) eigenvalue of the matrix $M_n(s)$.
 We have the sharp estimate (\cite{LS}, p.\ 152), the constants being optimal,
\begin{equation} \label{HS2} \frac{\zeta(2s)}{\zeta(s)^2}\le \l_n(s)\le \Lambda_n(s)\le \frac{\zeta( s)^2}{\zeta(2s)}  ,
\end{equation}
when $s>1$.
Consequently,
 \begin{equation}\label{boundquadra0}  \frac{\zeta(2s)}{\zeta(s)^2}\sum_{k=1}^n c_k^2\le \big\|\sum_{k=1}^n c_k f_k\big\|_2^2\le
\frac{\zeta( s)^2}{\zeta(2s)}\sum_{k=1}^n c_k^2  ,
 \end{equation}
when $s>1$. This implies that the series $\sum_k c_k f_k$ converges
in mean  if $\sum_k c_k^2<\infty$. In fact, it follows from
Gaposhkin's theorem cited above (see also \cite{Br}, p.\ 826) that
this series converges almost everywhere. Concerning eigenvalues,
Wintner (\cite{Wi}, p.\ 578) has shown that $\limsup_{n\to
\infty}\Lambda_n(s)<\infty$ if and only if $s>1$. Further, when
$1/2<s\le 1$,  it is known  that (\cite{LS}, p.\ 152)
\begin{equation} \label{HS3} \liminf_{n\to \infty}\l_n(s)=0, \qq  \limsup_{n\to \infty}\Lambda_n(s)=\infty.
\end{equation}

Our  first observation is that the  quadratic form (\ref{norm}) can
be for  $s>0$ more conveniently reformulated.
\begin{lemma} \label{02}Let $S=\{ x_1, \ldots, x_n\}$ be a set of distinct positive integers.
We assume that $S$ is     factor closed ($d|x_i\Rightarrow  d=x_j$
for some $j=1,\ldots, n$).   Let $s>0$. Then for  all   real  $y_1,
\ldots, y_n$,
$$\sum_{k,\ell=1}^ny_ky_\ell \frac{(x_k, x_\ell)^{2s}}{x_k^s x_\ell^s}= \sum_{i=1}^n J_{2s}
(x_i) \bigg\{ \sum_{k=1}^n {\bf 1}_{x_i|x_k} \frac{y_k}{x_k^s} \bigg\}^2,  $$
 where $J_\e =\xi_\e *\m    $ is the generalized Jordan totient
 function, $\xi_\e (k)= k^\e $ for all $k\in \Z$,   $\m $ being the M\"obius function.
    \end{lemma}
 \begin{proof} Immediate since  $(x_k, x_\ell)^{2s}=\sum_{ i=1  }^n   J_{2s} (x_i){\bf 1}_{x_i| x_k   } {\bf 1}_{x_i|  x_\ell } $.
   \end{proof}

\begin{remark}      Let $G$   and $A$ be  $n\times n$ matrices, with entries respectively defined by $(G)_{i,j}=  (x_i, x_j)^{2s} $ and
   $(A)_{i,j}= \sqrt{J_{2s} (x_i)} {\bf 1}_{x_i|x_j} $.
  By Lemma 4.1 of \cite{HL}, $ G= {}^t A A$.
  It follows that ${}^tUGU= {}^tV V$ where $V= AU$, namely
  $$ (V)_i= \sum_{k=1}^n  (A)_{i,k} u_k= \sqrt{J_{2s} (x_i)} \sum_{k=1}^n    {\bf 1}_{x_i|x_k} u_k.$$
  This is  another  more constructive way to obtain Lemma \ref{02}.
  \end{remark}
  The set $[1,n]$ being factor closed, by   taking $y_1=\ldots =y_{m-1}=0$ in the above Lemma, we get   for any $s>0$ and
  all   real  $y_m, \ldots, y_n$,
  \begin{equation}\Big\|\sum_{k=m}^n y_k f_k\Big\|_2^2=\zeta(2s)\sum_{i=1}^n J_{2s} (i) \bigg\{ \sum_{k=m}^n {\bf 1}_{ i| k} \frac{y_k}{k^s} \bigg\}^2.
  \end{equation}

The theorem below is Br\'emont's recent result (\cite{Br}, Theorem
1.1 (ii)) with only a slightly better formulation.  Let $\sigma_s (k)=\sum_{d|k} d^s$. 
\begin{theorem} \label{th_bremont} Let $1/2<s\le1$. Let $\p(k)>0$ and non decreasing.
Assume that both series  
$$ \sum_{k }  \frac{1}{k  \p (k)} , \qq \qq   \sum_{k }    c_k^2  \p( k)  (\log k)^2 \s_{1-2s}(k) $$
are convergent.
Then  the series $\sum_{k }  c_k f_k$ converges almost everywhere.
\end{theorem}
  By using Gronwall's estimates
(\cite{Gr} p.\ 119--122),
\begin{equation} \label{gron}\limsup_{x\to \infty}\frac{\s_1(x)}{x\log\log x}= e^\l, \qq  \limsup_{x\to \infty} \frac{\log
\big(\frac{\s_\a(x)}{x^\a}\big)}{\frac{(\log x)^{1-\a}}{\log\log x}} =\frac{1}{1-\a}, \qq (0<\a<1)
\end{equation}
where $\l$ is Euler's constant, and   the fact that  $ \s_{-\a} (x)= x^{-\a}
\s_\a(x)$, 
we easily recover (\ref{conds0}) and (\ref{conds1}). 
As we shall see later, the condition in Theorem
\ref{th_bremont} can still be weakened. \vskip 1pt

Br\'emont's proof is based on M\"obius orthogonalization and an
adaptation of Rademacher-Menshov's theorem. Schur's theorem and the
previous lemma allow to get it   shortly.
\begin{proof} Let $n\ge m\ge 1$. In the following calculation concerning
the norm $\big\|\sum_{k=m}^n c_k f_k\big\|_2$,  we may let  $c_k=0$
if $k\notin [m,n]$. Then
\begin{eqnarray*} 
\zeta(2s)^{-1}\big\|\sum_{k=m}^n c_k f_k\big\|_2^2&=&
   \sum_{i=1}^n J_{2s} (i) \bigg\{ \sum_{k=m}^n {\bf 1}_{ i| k} \frac{c_k}{k^s} \bigg\}^2.
\cr
 &\le &   \sum_{k,\ell=1}^n \frac{1}{k^s\ell^s}\bigg\{ \sum_{i=1}^n \frac{J_{2s} (i) }{i^{2s}}|c_{k i}||c_{\ell i} |   \bigg\}
\cr
&\le &   \sum_{k,\ell=1}^n \frac{1}{k^s\ell^s}\Big( \sum_{i=1}^n \frac{J_{2s} (i) }{i^{2s}}c_{k i}^2     \Big)^{1/2} \Big( \sum_{i=1}^n \frac{J_{2s} (i) }{i^{2s}} c_{\ell i}^2    \Big)^{1/2} \cr
&= & \bigg\{  \sum_{k =m}^n \frac{1}{k^s }
  \Big(\sum_{i=1}^n \frac{J_{2s} (i) }{i^{2s}}c_{k i}^2     \Big)^{1/2}     \bigg\}^2
\cr
&\le  & \bigg\{  \sum_{k =1}^n \frac{1}{k^s \psi(k)}   \bigg\} \bigg\{ \sum_{k =m}^n \frac{\psi(k)}{k^s  } \Big(\sum_{i=1}^n \frac{J_{2s} (i) }{i^{2s}}c_{k i}^2     \Big)     \bigg\}
   .
 \end{eqnarray*}
Now, choosing $\psi(k) = k^{1-s} \p(k)$, we have
\begin{eqnarray*}\sum_{k =m}^n  \frac{\psi(k)}{k^s  } \sum_{i=1}^n \frac{J_{2s} (i) }{i^{2s}}c_{k i}^2
&=& \sum_{1\le \nu \le n^2}c_\nu^2 \sum_{k|\nu\atop m\le k\le n} \frac{\psi(\nu/k)}{(\nu/k)^s  } \frac{J_{2s} (i) }{i^{2s}}
\cr
&\le & \sum_{m\le \nu \le n }c_\nu^2   \sum_{\k|\nu }  \p (\k)  \k ^{1-2s}
\le  \sum_{m\le \nu \le n }c_\nu^2  \p( \nu) \s_{1-2s}(\nu) .
\end{eqnarray*}
  By combining, we find for all $n\ge m\ge 1$,
   \begin{eqnarray} \label{HS1}  \big\|\sum_{k=m}^n c_k f_k\big\|_2^2
&\le  &  \zeta(2s)\bigg\{  \sum_{k =1}^n \frac{1}{k  \p (k)}   \bigg\} \bigg\{ \sum_{k =m}^n c_k^2  \p( k) \s_{1-2s}(k)     \bigg\}  ,
 \end{eqnarray}

Assume that the series $\sum_{k }  \frac{1}{k  \p (k)} $ converges. Let $\tilde f_k=   f_k/( C_{s,\p}\p( k) \s_{1-2s}(k))^{1/2}$ where $C_{s,\p}=\zeta(2s)   \sum_{k \ge 1}  \frac{1}{k  \p (k)} $. Then
 $ \big\|\sum_{k=m}^n c_k \tilde f_k\big\|_2^2
 \le     \sum_{k =m}^n c_k^2 $.
 In view of Schur's Lemma  (Lemma \ref {schur}), this implies that these functions   can be extended   to an
orthonormal system on a   bounded interval $X$ of the real line including $[0,1[$,  and  endowed with the normalized Lebesgue measure.  
 By  Rademacher-Menshov's theorem, the series $\sum_{k }  c_k \tilde f_k$ converges almost everywhere  once the series $ \sum_{k }    c_k^2  (\log k)^2$
converges. Equivalently,  the series $\sum_{k } \g_k  f_k$ converges almost everywhere  once the series $  \sum_{k }    \g_k^2  \p( k)  (\log
k)^2\s_{1-2s}(k) $  converges, as claimed.
  \end{proof}
  \begin{remark}     The monotonicity of the $L^2$ norm with respect to Fourier coefficients yields that Theorem \ref{th_bremont}
  remains valid for $f(x)= \sum_{j=1}^\infty a_j \sin 2\pi jx $, assuming that $|a_j|=\mathcal O(j^{-s})$, $1/2<s\le 1$.
    \end{remark}
  \begin{remark} \label{more} Schur's Lemma implies much more. The series $\sum_{k }  c_k f_k$ converges almost everywhere for any
 coefficient sequence $\{c_k, k\ge 1\}$ such that 
\begin{eqnarray}  \label{cconv1}
 \big\{c_k  \big(\p( k)   \s_{1-2s}(k)    \big)^{1/2}  , k\ge 1\big\}
\end{eqnarray} is   universal, according to Definition \ref{universal}. 
 
\end{remark}    \begin{remark} \label{gal}  Estimate (\ref{HS1}) indicates that 
\begin{eqnarray*}   
 \Big|\sum_{k,\ell =m}^n c_kc_\ell   \frac{(k,\ell)^{2s}}{k^s\ell^s}\Big|
&\le  &   C(\p)    \sum_{k =m}^n c_k^2  \p( k) \s_{1-2s}(k)     , \end{eqnarray*}     
 assuming $C(\p)=\sum_{k =1}^\infty \frac{1}{k  \p (k)} <\infty $. Take $s=1$ and  let 
 $ 1\le \k_1<\k_2<\ldots <\k_r   $
be integers. Choose $m=\k_1$, $n=\k_r$ and
  $c_k=1$, if $k=\k_j$ for   some $1\le j\le r$, and 
$c_k=0$ otherwise.
 Letting  also $\p(x)= (\log x) (\log\log x)^{1+\e}$, $\e>0$, and using   (\ref{gron}), we find that
 $$   
 \sum_{ i,j=1}^r  \frac{(\k_i,\k_j)^{2 }}{\k_i \k_j } 
 \le     C_\e   \sum_{i =1}^r   ( \log  \k_i) (\log\log k_i)^{2+\e} . $$
 However, this is   far from being optimal. G\'al's estimate (\cite{G}, Theorem 2) indeed implies
\begin{equation}\label{boundquadras}   \sum_{ i,j=1}^r  \frac{(\k_i,\k_j)^{2 }}{\k_i \k_j }  \le
Cr(\log\log r)^2 .
 \end{equation} 
  \end{remark}  

\vskip 3pt

Before going further,  recall that  by Wintner's fundamental theorem, the
series $\sum_{n=1}^\infty c_n f(nx)$ converges in the mean for all
${\bf c}\in \ell^2$ iff
\begin{equation}\label{wintner} \sum_{n=1}^\infty a_n /n^s \quad \hbox{and} \quad \sum_{n=1}^\infty
b_n /n^s\quad \hbox{are regular and  bounded for\ } {\Re}  s>0.
\end{equation}
 The following result (\cite{BW}, Theorem 3.1) concerns the situation when condition  (\ref{wintner}) fails.
\begin{theorem}\label{bw.div} Let $f\in \rm{Lip}_\alpha ({\bf  T})$,
$0<\alpha\le 1$, $\int_{\bf  T} f(t)dt=0$ and assume that (\ref{wintner}) is
not valid. Then for any $\varepsilon_k \downarrow 0$ there exists
${\bf c}\in \ell^2$ and a sequence $(n_k)$ of
positive integers satisfying
$$ n_{k+1}/n_k \geq 1+\varepsilon_k \qquad (k \geq k_0 ) $$
such that the series $\sum_k c_k f(n_k x)$ is a.e.\  divergent.
\end{theorem}
\begin{remark}  If $(n_k)$ grows exponentially (i.e.\
$n_{k+1}/n_k \ge q >1$), then $\sum_{k=1}^\infty c_k f(n_kx)$
converges a.e.\ for any ${\bf c}\in \ell^2$ by Kac's theorem \cite{Ka}. Thus Theorem \ref{bw.div} is sharp. It also remains true with minor modifications
in the proof, if instead of $f\in \rm{Lip}_\alpha ({\bf  T})$ we
assume only $f\in L^2({\bf  T})$. For the class of functions defined in
(\ref{funct}), Br\'emont has recently showed a similar result in \cite{Br}.
\end{remark}

      \vskip 7pt  In   the general case  the following   quadratic form
      appears:
$$\sum_{k, \ell\in
K } c_k c_\ell \frac{  ( k,  \ell)}{  \ell\vee  k}.
$$
Since $\frac{  ( k,  \ell)}{  \ell\vee  k}\le \frac{  ( k,  \ell)}{
\sqrt{k\ell  }  }$, this
 may be regarded as a continuation of  the limiting case $s=1/2$. Recall some basic facts concerning quadratic forms. Let   $U_n$ denote the unit sphere of
$\R^n$ and let $A=\{a_{i,j}\}_{i,j=1}^n$ be an
$n\times n$ real  symmetric matrix
  with characteristic roots
$\l_1,\ldots, \l_n$.   It is well-known that the set of values
assumed by the quadratic form $Q({\bf x})=\sum_{i,j=1}^n
x_ix_ja_{i,j}$ when ${\bf x}=(x_1,\ldots, x_n)\in U_n$, coincides
with the set of values assumed by $\sum_{i=1}^n \l_iy_i^2$ on $U_n$.
See \cite{B}, p.\ 39 and Chapter 4.

 Hence we get
 \begin{equation}   \big( \inf_ {i=1}^n \l _i\big)\sum_ {i=1}^n x^2_i\le
\big|Q({\bf x})\big|\le  \big( \sup_ {i=1}^n \l_i\big)\sum_ {i=1}^n x^2_i  . \end{equation}
 This   way to estimate $Q({\bf x})$ strongly relies  on a good knewledge of the extremal eigenvalues. The classical weighted
estimate     below is often more convenient.
\begin{lemma} \label{quadraf}   For any system of complex numbers $\{x_i\}$ and  $ \{\a_{i,j}\}$,
 $$   \Big|  \sum_{ 1\le i,j\le n\atop i\not= j}  x_ix_j\a_{i,j} \Big|\le  \frac1{2}\sum_{ i=1}^n x_i^2\Big(  \sum_{\ell
=1\atop
\ell\not = i }^n (|\a_{i,\ell} | +|\a_{
\ell ,i} |) \Big)  .
$$
   \end{lemma}
 \begin{proof} We have
$$\Big| \sum_{1\le i<j\le n}  x_ix_j\a_{i,j}\Big|\le   \sum_{1\le i<j\le n}  \big(\frac{x_i^2+x_j^2}{2} \big)|\a_{i,j} |\le \frac1{2}
\sum_{i=1}^n x_i^2 \Big(\sum_{i<\ell\le n } |\a_{i,\ell}
|+\sum_{1\le  \ell <i }  |\a_{ \ell ,i} |\Big)  .$$ Operating
similarly for the sum $\sum_{1\le j<i\le n}  x_ix_j\a_{i,j}$ gives
the  result.
  \end{proof}

 This suggests to   attach to
$K$ a function
$\t_K $ defined   by
\begin{equation}\label{cac}\t_K(k)=\sum_{\ell\in K\atop
\ell\not =k} \frac{ (\ell,k)}{  \ell\vee  k} .
\end{equation}
 The associated
coefficient
 $$\t_K =\sup_{k\in K }\t_K(k)$$
will serve as a  measure of   the arithmetical complexity of $K$. For instance, $\t_K$ is small if $K$ is a set of prime numbers,
and uniformly
bounded over all subsets
$K$ of a given chain, as we shall see. A sequence $\mathcal N=\{n_k, k\ge 1\}$ is a called a chain if $n_k|n_{k+1}$ for all $k$.
    There are    examples for which $\t_K(k)=o(1)$,
if  $k$ is  large.
\vskip 4pt \noi{\it Example 1.} {\gsec  (Prime sequence)}
  Take
$K=\mathcal P \cap [N/2, N]$    where  $\mathcal P$ denotes the sequence of consecutive primes. Let $\pi(n)=
\sharp\{  p\, {\rm prime} , p\le n\}$ be the prime numbers function.  Then
$$\sum_{N/2\le \ell <k\atop \ell \in \mathcal P}\frac{(\ell, k)}{\ell \vee k}\le \frac{\pi(k) }{  k}\le \frac{C }{ \log  k},$$
and$$\sum_{k <\ell \le  N\atop \ell \in \mathcal P}\frac{(\ell, k)}{\ell }= \sum_{k <\ell \le  N\atop \ell \in \mathcal
P}\frac{1}{\ell }\le 2\frac{\pi(N) }{N}\le \frac{C }{
\log  N} \le \frac{C }{ \log  k}.$$  So that for all $k\in K$,
\begin{equation}\label{cac2}\t_{\mathcal P \cap [N/2, N]}(k) \le  {C }/{ \log  k}.
\end{equation}

It is easy to extrapolate from this  that $\t_K$ can be on examples as small as wished. There are also important classes of sequences for
which
$\t_K
$ is uniformly bounded over all of its finite parts
$K$.
\vskip 3pt\noi{\it Example 2.} {\gsec  (Hadamard gap  sequences)}
Consider  a  sequence $\mathcal N=\{n_k, k\ge 1\}$ satisfying the
Hadamard gap condition
\begin{equation} \label{hgc} \frac{n_{k+1}}{n_k}\ge q>1.
\end{equation}     Let  $K=\{ n_k, k\in \mathcal K\}$. Then
 \begin{equation} \label{hgc1}  \sup_{\mathcal K}\sup_{\ell \in \mathcal K} \t_{\mathcal K}(\ell)= \tau (q)<\infty.
\end{equation}
Indeed, for $\ell\in \mathcal K$,
$$ \sum_{k\in \mathcal K\atop
k< \ell}   \frac{  (n_k, n_\ell)}{ n_\ell\vee n_k } \le  \sum_{  k<\ell}  \frac{   n_k\wedge
n_\ell }{ n_\ell\vee n_k } =\sum_{  k<\ell} \frac{   n_k   }{ n_\ell  }\le \sum_{k<\ell }  q^{-(\ell-k)} \le C_q<\infty. $$
Similarly
 $\sum_{k\in \mathcal K\atop
k> \ell}   \frac{  (n_k, n_\ell)}{ n_\ell\vee n_k }\le  \sum_{
k> \ell}   \frac{   n_k\wedge n_\ell }{ n_\ell\vee n_k }=\sum_{  k>\ell} \frac{   n_\ell   }{ n_k  }\le \sum_{k>\ell }  q^{-(k-\ell )} \le
C_q<\infty$.  As $C_q$ depends on $q$ only. This yields (\ref{hgc1}).
  \vskip 3pt \noi{\it Example 3.}\,{\gsec(Squarefree numbers)}
Let   $\mathcal G$ be the set of squarefree numbers generated by
some increasing sequence $2\le p_1<p_2<\ldots$ of  prime integers
satisfying the following condition
\begin{equation}\label{sqc} \m= \sum_{i=1}^\infty \frac{1}{p_i}<1.\end{equation}
 Take $K\subset \mathcal G$. Writing in what follows $\ell= \l d$, $k=\k d$ with $d=(\ell,k)$, we easily get
\begin{equation}\label{ }\t_K(k)= \sum_{\ell\in K\atop
\ell\not =k} \frac{ (\ell,k)}{  \ell\vee  k}\le \frac{1}{  \k}\sum_{
\l  \le \k}  1+\sum_{\l\in \mathcal G\atop
\l>k} \frac{1}{  \l } \le C +\sum_{ p_i>k} \frac{1}{  p_i }+\sum_{ p_ip_j>k} \frac{1}{  p_ip_j } +\ldots,
\end{equation}
since the number of squarefree integers less than $x$ is of order $  6x/\pi^2$. Now
$$\sum_{ p_i>k} \frac{1}{  p_i }+\sum_{ p_ip_j>k} \frac{1}{  p_ip_j } +\ldots\le \m+\m^2 +\ldots<\infty.$$
Hence
$$\sup_{K\subset \mathcal G }\t_K<\infty.$$
\vskip 7pt A first basic mean estimate  obtained in this work (the
proof will be given in Section~\ref{ntt}) is the following.
   \begin{lemma} \label{basic0}     For any finite sequences of reals $\{a_j, j\in J\}$, $\{c_k, k\in K\}$,  $$ \big \|\sum_{ k \in
K}c_k\sum_{j\in J} a_je_{kj}\big\|_2^2
  \le   C |J|  \sup_{j\in J}  |a_j|^2    \sum_{k \in K } c_k ^2  \max(1,\t_{K }(k))  .
 $$
\end{lemma}
 \vskip 7pt    A general bound for $\t_K(k)$ can be provided by using Pillai's arithmetical function $P(k)$, which is defined by $P(k)=
\sum_{d=1}^k (d,k)$. Recall that we have $P(k)= \sum_{d| k}d \phi(k/d)$, so that the arithmetic mean of $(1, k),\ldots,  (k, k)$ is given
by
  \begin{equation}\label{A} A(k)=\frac{P(k)}{k}=\sum_{\k|k}\frac{\phi(\k)}{\k}.
\end{equation}
   \begin{lemma}   \label{gb}
 For all finite sets $K$ of integers and all
$  k\in K
$,
\begin{equation}\label{cac1}\t_K(k)\le C\log (\frac{eK_+ }{k})  A(k) .
\end{equation}
where    $C$ is an absolute constant, and   $K_+$ (resp.
 $K_{-}$) denotes   the largest (resp. smallest) term of $K$.
\end{lemma}
This follows immediately from  Lemma \ref{basic01} below. Example 1
shows that    (\ref{cac1}) is   not always optimal. Estimate
(\ref{cac1}), however, implies that if $\mathcal M=\{m_k, k\ge 1\}$
is a sequence of mutually coprime integers,   then
$$\sup_N \sup_{K\subset [\rho N, N]}\t_K =
C_\rho <\infty.
$$

\begin{remark}    The main orders of $A(k)$ are well known. As  $C\frac{k}{\log\log k}\le \phi(k)\le k$, the function $A(k)$ always satisfies
$$\frac{d(k)}{\log\log k}\le A(k)\le d(k) ,$$
where   $d(k)$ denotes the number of divisors of
$k$. As to the maximal order,  we have Chidambaraswamy and Sitaramachandrarao estimate,
 \begin{equation} \label{pillai0}  \limsup_{n\to\infty} \frac{\log A(n)\log\log n}{\log n}= \log 2.
\end{equation}
This   is well-known for the function $d(n)$ instead of $A(n)$.  We   refer  to T\'oth's  recent survey   \cite{To} on  Pillai's
function.
\end{remark}  \vskip 4pt

 For the class of examples previously considered, we have the following  \begin{proposition}\label{prop01}  Let $K$ be a finite set of integers. For any $k\in K$,
\begin{eqnarray*} \sum_{\ell\in K\atop
\ell\not =k}  \frac{(k,\ell)^{2s}}{k^s\ell^s}\le \begin{cases}
2\big(\log \frac{K_+  } {K_- } \big)\s_{-1}(k)   &\quad {\rm if}\ s =1,
\cr C_s 2^sk ^{s-1}\big(\int^{K_+  }_{K_-  }\frac{\dd
u}{u^s}\big)
\s_{ 1-2s}(k)   &\quad {\rm if}\ s<1. \end{cases}
\end{eqnarray*}
 \end{proposition}
Thus
 $ \sum_{\ell\in K\atop
\ell\not =k}  \frac{(k,\ell)^{2s }}{k^s \ell^s }\le   2^s M  ^{1-s}
\s_{ 1-2s}(k) ,$ if $K_+\le MK_-  $.
\begin{proof}
  Let $s=1 $. As  for $\l \ge 1$,
  $\frac{1}{\l }\le \min\big(  \int_{\l-1}^\l\frac{\dd t}{t} ,2\int^{\l+1}_\l\frac{\dd t}{t})$,
 we have
\begin{eqnarray*} \sum_{\ell\in K\atop
\ell\not =k}  \frac{(k,\ell)^{2 }}{k \ell }&\le & \sum_{d|k }  \frac{1}{(k/d)} \sum_{ K_- /d\le \l \le K_+/d\atop
\ell\not =k}  \frac{1}{\l }
\cr &=&  \sum_{d|k }  \frac{1}{(k/d)}\bigg\{ \sum_{ K_- /d\le  \l <k/d } \frac{1}{\l } +\sum_{ k/d< \l \le K_+/d }  \frac{1}{\l }\bigg\}
\cr &\le &  \sum_{d|k }  \frac{1}{(k/d)}\bigg\{ 2\int_{K_- /d }^{k/d}\frac{\dd t}{t} +\int_{k/d }^{K_+/d}\frac{\dd t}{t}\bigg\}
\cr &= &  \sum_{d|k }  \frac{1}{(k/d)}\bigg\{ 2\int^{k  }_{K_-  }\frac{\dd
u}{u} +\int_{k  }^{K_+ }\frac{\dd
u}{u}\bigg\}\le  2\sum_{d|k }  \frac{1}{(k/d)} \int^{K_+  }_{K_-  }\frac{\dd
u}{u}
.
\end{eqnarray*}
 Similarly, when $0<s<1$,
\begin{eqnarray*}\sum_{\ell\in K\atop
\ell\not=k}  \frac{(k,\ell)^{2s}}{k^s\ell^s}&\le&
  \sum_{d|k }  \frac{1}{(k/d)^s}\bigg\{ \sum_{ K_- /d\le  \l <k/d } \frac{1}{\l ^s} +\sum_{ k/d< \l \le K_+/d }  \frac{1}{\l^s }\bigg\}
\cr &\le &
  \sum_{d|k }  \frac{1}{(k/d)^s}\bigg\{ 2^sd^{s-1}\int^{k  }_{K_-  }\frac{\dd
u}{u^s}+d^{s-1}\int_{k  }^{K_+ }\frac{\dd
u}{u^s}\bigg\}
\cr &\le  & 2^s\bigg\{  \int^{K_+  }_{K_-  }\frac{\dd
u}{u^s} \bigg\} \sum_{\k|k }  \frac{ (k/\k) ^{s-1}}{\k^s}
= 2^sk^{s-1}\bigg\{  \int^{K_+  }_{K_-  }\frac{\dd
u}{u^s} \bigg\}  \s_{1-2s}(k) .
\end{eqnarray*}
\end{proof}
 This implies when combined with Lemma \ref{quadraf}, if
$K_+ \le C K_{-}$,
  \begin{eqnarray}\label{61}\qq    \big\|\sum_{k\in K} c_k f_k\big\|_2^2  \le
  C_s  \sum_{k\in K} \s_{1-2s}(k) c_k^2
  ,  \end{eqnarray}
when $1/2< s\le 1$, which  is slightly more precise than (\ref{HS1}).  In the case $s=1/2$,  not covered by the class of functions   (\ref{funct}), it also gives
 \begin{equation} \label{demi} \sum_{k,\ell \in K} c_kc_\ell \frac{(k,\ell)}{\sqrt{k\ell  }}\le C\sum_{k\in K} d(k)c_k^2.
\end{equation}

\vskip 7 pt

\section{\bf   Main Results}  \label{mr}

\medskip
 We now state  the main results of this paper. We first consider mean convergence. Let
$f\in L^2$. Define for $t>0$, and any sequence $\uc=\{c_k, k\ge 0\} $     of reals,
$$S_t(\uc)=\sum_{k\in \mathcal N\atop k\le t} c_k f_k . $$
\begin{theorem}\label{square} Let $f\sim \sum_j a_je_j $ and assume that the following condition is satisfied:

 For some real $M>1$,
\begin{equation}\label{condaj121}L=\sum_{v=0}^\infty
M^{v }
\big( \sup_{M^v \le j< M^{v+1 } }  a_j^2\big)<\infty.
\end{equation}

  a) Let $\mathcal N=\{n_k, k\ge 1\}$ be an increasing sequence of positive integers
satisfying for any $\m>1$,\begin{eqnarray}\label{hyp1} \sup_{j\ge
0}\t_{\mathcal N\cap [\m^j, \m^{j+1}[}<\infty  .
\end{eqnarray}
Then there exists a constant $C$ such that for any $\uc \in \ell_2$,
 $$\Big(\sum_{j\ge 0} \big\|S_{\m^{j+1}}(\uc)-S_{\m^{j }}(\uc)\big\|_2^2\Big)^{1/2}\le C\|\uc\|_2. $$
 b) Assume that for any $\m>1$,
\begin{eqnarray}\label{hyp2}
 \t_{\mathcal N\cap [\m^j, \m^{j+1}[} =o(1)  \qq\qq j\to \infty.
\end{eqnarray}If the coefficient sequences $\ua$, $\uc$ have each constant signs, then
$$\|\uc\|_2\le \Big(\sum_{j\ge 0} \big\|S_{\m^{j+1}}(\uc)-S_{\m^{j }}(\uc)\big\|_2^2\Big)^{1/2}\le C\|\uc\|_2. $$
 \end{theorem}

\begin{remark}
By  (\ref{cac1}), condition (\ref{hyp1}) is satisfied as soon as
$$ \sup_{  k\in \mathcal N}A(k)<\infty  .
$$
\end{remark}

 \vskip 6 pt  We also establish   new almost everywhere convergence results.
  \begin{definition}\label{universal} We say that  a sequence of coefficients
$\uc $ is universal if for any orthonormal system $\Phi $ on a
bounded interval, the
 series $\sum_{k=1}^\infty c_k\p_k $ converges a.e.
  \end{definition}
 Typically,   $\uc
$ is universal if the series  $\sum_k c^2_k \log^2 k$ converges (Rademacher-Menshov
theorem), or if the series
 $\sum_k c_k^2  ( \log   |c_k|^{-1} )^{1+h}
 ( \log k )^{1-h}$ converges for some $0\le h<1$ (Tandori's
theorem \cite{Ta2}).
And   the condition $\sum_{k} c_k^2(\log  |c_k|^{-1})^2<\infty $, with $c_k\not=0$,  $c_k\to 0$   is  necessary  for $\uc
$ to be  universal, see \cite{Ta1}.

    \begin{theorem} \label{thN} Assume that there exist a non-increasing sequence of positive reals $\{\e(j), j\ge 1\}$   and an
increasing sequence of positive integers $\{j_r, r\ge 1\}$,   such that
\begin{eqnarray}  \label{condaj1}   A= \sum_{|a_\ell|>\e(\ell)}    |a_\ell |   <\infty,
 \qquad  B=\sum_r j_{r+1}^{1/2}
\e(j_r) <\infty.
 \end{eqnarray}
 Let  $1\le k_1<k_2<\ldots $ be an  increasing sequence of integers, which we denote by $K$. Then the series $\sum_{n\ge 1}
c_{ n} {f _{k_n}}$ converges a.e. for any
 coefficient sequence $\{c_n, n\ge 1\}$ such that
\begin{eqnarray}  \label{cconv1}
 \big\{c_n  \max\big(1, \t_K({k_n} )^{1/2}  \big) , n\ge 1\big\}
\end{eqnarray} is   universal.
\vskip 2pt
b)   In particular, the same conclusion holds if
\begin{eqnarray}  \label{condaj12}   A= \sum_{|a_\ell|>\e(\ell)}    |a_\ell |  <\infty,
 \qquad     B_1=\sum_{j}    \e^2(j) <\infty.
 \end{eqnarray}
in place of (\ref{condaj1}).  \end{theorem}

\begin{remark}\label{thNr}

\medskip
\item (i) If condition   (\ref{condaj1}) is satisfied  for   $j_r=M^r$, for some $M>1$, then $B<\infty$
means $ \sum_r M^{r/2 } \e(M^r) <\infty$, which is a stronger
requirement than $ \sum_r M^{r  } \e^2(M^r) <\infty$. And this is
equivalent to $B_1<\infty$ in (\ref{condaj12}). Hence
(\ref{condaj1}) can be replaced by the much weaker condition
(\ref{condaj12}) when $j_r $ is geometrically growing. \vskip 3pt
\item(ii) Suppose now  that $f$ satisfies   assumption (\ref{condaj121}).
    Then (\ref{condaj12}) is fulfilled. Indeed, choose
$$\e_\ell= \sup_{M^r\le |j|\le M^{r+1}}|a_j|, \qq  M^r\le \ell\le M^{r+1}, \ r=0,1,\ldots$$
The first requirement in (\ref{condaj12}) is   trivially satisfied since the summation index is empty, whereas  the second
  is, by (i), equivalent to (\ref{condaj121}).
\vskip 3pt
\item(iii) Let $f\in \hbox{Lip}_\alpha(\T)$, $\alpha>1/4 $. Then $f$ satisfies condition (\ref{condaj12}). Indeed, it is well-known that if
$f\in \hbox{Lip}_\alpha(\T)$, $0<\alpha\le 1 $,   then $\sum_{2^r
<j\le 2^{r+1} }a_j^2\le C2^{-2r\a }  $. See \cite{Z}, inequalities
(3$\cdot$3) p.\ 241. Pick  a real $\b$ such that $2\a>\b>1/2$ and
take $\e(j)= j^{-\b}$, $M=2$. Condition (\ref{condaj12}) is
satisfied with this choice since $\sum_j\e^2(j)<\infty$ and
$$\sum_{2^r <j\le 2^{r+1}\atop
|a_j|>\e(j)} |a_j| \le \sum_{2^r <j\le 2^{r+1} } \frac{|a_j|^2}{\e(j)} \le C2^{ r\b }2^{-2r\a } = C 2^{- r(2\a-\b) },$$
so that $A<\infty$.
\vskip 3pt
\item(iv) For any
$\a>0$, there exists $f\in L^2(\T)$, $\int_\T f=0$, such that
$f\notin \hbox{Lip}_\alpha(\T)$ but $f$ satisfies condition
(\ref{condaj12}). Such an $f$ can be built as follows. Let
$\psi:\R^+\to \R^+$ be such that
\begin{eqnarray*}\begin{cases} \psi(r)2^{-r/2}\downarrow 0\  , \  \psi(r)2^{\a r}\uparrow\infty\quad {\rm as} \ r\uparrow
\infty,\cr
\sum_r \psi^2(r)<\infty.\end{cases}
\end{eqnarray*}
   Let $\e:\R^+\to \R^+$ be decreasing and defined by
\begin{eqnarray*}\e(x)=\begin{cases}  \psi(r)2^{-r/2}\quad {\rm if} \ 2^r<j\le 2^{r+1}, \quad r \ {\rm even},\cr
 {\rm linear\  otherwise}.\end{cases}
\end{eqnarray*}
We choose $f$ such that its Fourier coefficients satisfy
\begin{eqnarray*} \begin{cases}  a_j=\psi(r)2^{-r/2}\quad {\rm if} \ 2^r<|j|\le 2^{r+1}, \quad r\ {\rm even},\cr
\displaystyle{\sum_{r\, {\rm odd}}\sum_{2^r<|j|\le 2^{r+1}} |a_j|<\infty}.\end{cases}
\end{eqnarray*}
Clearly
$$ \frac{\sum_{2^r<|j|\le 2^{r+1}} |a_j|}{2^{r(\frac{1}{2}-\a)}}=C\frac{\psi(r)2^{ \frac{r}{2}}}{2^{r(\frac{1}{2}-\a)}}=\psi(r)2^{\a r}
\uparrow\infty.$$ Hence, in view of \cite{Z}, inequality (3$\cdot$4)
p.\ 241, $f\notin \hbox{Lip}_\alpha(\T)$. Further,
$$\sum_{2^r<|j|\le 2^{r+1}} |a_j|^2=  2^r\psi^2(r)2^{-r}= \psi^2(r),$$
when $r$ is even. Thus
$$\sum_{j\in \Z
*}  |a_j|^2=\sum_{r\ge 0}\sum_{2^r<|j|\le 2^{r+1}} |a_j|^2\le \sum_{r\, {\rm even}}\psi^2(r)+
\sum_{r\, {\rm odd}}\sum_{2^r<|j|\le 2^{r+1}} |a_j|   <\infty,$$ by assumption. Moreover, by construction,
$$\sum_{2^r<|j|\le 2^{r+1}\atop |a_j|>\e(j)} |a_j|=0,$$
when $r$ is even. It follows that
$$ \sum_{  |a_j|>\e(j)} |a_j|= \sum_{r\, {\rm odd}}\sum_{2^r<|j|\le 2^{r+1}\atop |a_j|>\e(j)} |a_j|\le \sum_{r\, {\rm
odd}}\sum_{2^r<|j|\le 2^{r+1} } |a_j| <\infty, $$ by assumption.
Now
$$ \sum_{j} \e(j)^2\le 3 \sum_{r\, {\rm even}} 2^r\psi^2(r)2^{-r}=3 \sum_{r\, {\rm even}}  \psi^2(r)  <\infty.$$
Therefore condition (\ref{condaj12}) is satisfied, as claimed.
\end{remark}\vskip 5pt
  We will also obtain the following useful result, as a combination of the above Theorem with Lemma
\ref{gb}.

\begin{corollary}\label{cNc}
Let $1\le k_1<k_2<\ldots $ be an  increasing sequence of integers.
Assume that (\ref{condaj12}) is satisfied and that
\begin{eqnarray}\label{cconv2}   \sum_n c_n^2A ({k_n} )(\log n)^2<\infty.
\end{eqnarray}
 Then the series  $ \sum_{n}   c_n  f_{k_n} $ converges a.e. \end{corollary}
By Remark \ref{thNr}-ii),  the same conclusions are reached if  $f$ satisfies   assumption (\ref{condaj121}).

 \begin{remark} \rm
     As $\  A(k)\le
d(k)$,   (\ref{cconv2}) is satisfied whenever
\begin{eqnarray}  \label{cconv4}  \sum_{ n
  } c_n^2d( k_n)  (\log n)^2    <\infty.
\end{eqnarray}
Consequently, under  condition (\ref{cconv4})  the series
$\sum_{n\ge 1} c_{ n} {f _{ n}}$ converges a.e. for any $f\in
\hbox{Lip}_\alpha(\T)$, $\alpha>1/4 $. The     presence of the
factor $d(k_n)$   is   important. Replacing $d(j)$ by the classical
bound:  for some  $c_0>2$,
  \begin{equation} \label{divbound} d(j) = {\mathcal O}\big( c_0^{\log j/\log\log j}\big),
\end{equation}
   gives rise to a much weaker result. This strictly includes  a recent result obtained by Aistleitner \cite{A}
   who proved by using a fine
diophantine estimate due to Dyer and Harman, that the condition
\begin{equation*} 
\sum_{k=1}^\infty c_k^2 \exp\Big(\frac{2\log k}{\log\log k}\Big)<\infty,
\end{equation*}
is   sufficient for the   convergence almost everwhere of the series
$\sum_{k=0}^\infty c_k f( kx)$. It is interesting to compare the
multiplicative factor of $c_k^2$ in the above with the shape of the
bound of the divisor function in (\ref{divbound}). This can also be
deduced from Theorem 1.1 in    \cite{W1} published shortly
afterward, and which was based on properties of  the Erd\"os-Hooley
function
$$\D(v)=\sup_{u\in \R}\sum_{d|v \atop
x<d\le ex} 1.$$
In place of condition (\ref{condaj1}), we assumed that $f$ satisfies
\begin{equation} \label{hooley} \sum_{\nu\ge 1 }  a_{ \nu }^2\D(  \nu  ) <\infty.
\end{equation}
This is fulfilled if $f\in \hbox{Lip}_\alpha(\T)$, $\alpha>1/4 $,
but also   if $a_\nu= \mathcal O  (\nu^{-\b})$, $\b>1/2$. Conditions
(\ref{condaj1}) and (\ref{hooley}) are, however, hardly  comparable.
As is well known,     $\D$  has  a
   slower   mean behavior than $d$. Indeed,
  $$\frac1{x}  \sum_{v\le x} d(v) \sim x,\qq {\rm while}\qq \frac1{x}  \sum_{n\le x} \D(n)= \mathcal O\Big( e^{c \sqrt{\log\log x \cdot
\log\log\log  x}}
\Big) $$ for a suitable constant $c>0$; see \cite{T1}.
Hence it follows by   partial summation that if $f$ has   monotonic Fourier coefficient sequence, condition (\ref{hooley}) can be
replaced by the considerably much weaker condition
\begin{equation} \label{hooley1} \sum_{\nu\ge 1 }  a_{ \nu }^2e^{c \sqrt{\log\log \nu \cdot
\log\log\log  \nu} }<\infty.
\end{equation}
 \end{remark}

When $|a_j|=\mathcal O(j^{-s})$, $s>1/2$, the above corollary can
be much improved.
\begin{theorem}\label{cNc1}
Let $1\le k_1<k_2<\ldots $ be an  increasing sequence of integers.
Let $f(x)= \sum_{j=1}^\infty a_j \sin 2\pi jx $   and assume that
$|a_j|=\mathcal O(j^{-s})$, $s>1/2$.
 Assume that
 \begin{eqnarray} \label{cconv2}
 \sum_n c_n^2\s_{1-2s}({k_n} )(\log n)^2<\infty.
\end{eqnarray}
 Then the series  $ \sum_{n}   c_n  f_{k_n} $ converges a.e. \end{theorem}
     Our paper is organized as follows. In Section   \ref{ntt}, we  collect    estimates  of number theoretical type,  some estimates for quadratic forms  and   tools from the theory of orthogonal sums. The remainding sections are devoted to the proofs of the main results.


\bigskip
 \section{\bf Auxiliary Results}  \label{ntt}

\medskip
    \begin{lemma} \label{basic01} For  any positive integers $k \le N$,
$$\sum_{1\le \ell \le  N\atop \ell\not= k}\frac{(\ell, k)}{\ell \vee k} \le C\log (\frac{eN}{k})  \sum_{\k |k}\frac{\phi(\k)}{  \k} , $$
where $C$  is an absolute constant.
  \end{lemma}
    \begin{proof}
    Let $k <\ell \le  N$.  Then,
 $$\sum_{k <\ell \le  N}\frac{(\ell, k)}{\ell }\le \sum_{d|k}\sum_{k/d<\l < N/d\atop (\l, k/d)=1}\frac{1}{ \l } , $$
where we write      $\ell= \l d$, $k=\k d$, $(\ell, k)=d$.
 To estimate the inner sum, we  use van Lint and Richert estimate (\cite{LR}, Lemma 2): for $x\ge 1$ and $k$ such that $P^+ (k)\le
x$, where
$P^+ (k)$ is the largest prime factor of $k$, we have
\begin{equation}\label{LRest} \sum_{1\le m\le x \atop (m,k)=1}1\le C \frac{\phi(k)}{k}x.
\end{equation}
   By   \cite{T2} p.\ 3, if  $a_n$ are complex numbers,
$A(t)=\sum_{n\le t}a_n$ and
$b\in
\mathcal C^1([1,x])$,
\begin{equation}\label{simple}\sum_{1\le n\le x} a_nb(n)= A(x)b(x)-\int_1^x A(t) b'(t) dt,
\end{equation}
  Take $a_\l=0$ if $1\le \l< k/d$ and $a_\l= \chi\{ (\l, k/d)=1\}$ if $\l\ge k/d$, $b(t)=t^{-1}$. Then
$A(t)=0$ if $t< k/d$. Now if $k/d\le t\le N/d$, obviously $  P (k/d)\le t$. And   (\ref{LRest}) applies to give
$$ A(t) \le C\frac{\phi(k/d)}{(k/d)}t.$$
   Therefore
 \begin{eqnarray}\label{bba}
\sum_{k/d<\l < N/d\atop (\l, k/d)=1}\frac{1}{ \l } &= & \frac{A(N/d)}{(N/d) } + s \int_{k/d}^{N/d}
A(t)\frac{\dd t}{t^{2}}\cr
&\le &  \frac{A(N/d)}{(N/d) } +   C   \frac{\phi(k/d)}{(k/d)} \int_{k/d}^{N/d}
 \frac{\dd t}{t } \cr
&\le &C \Big(\frac{1 }{(N/d) }\frac{\phi(k/d)}{(k/d)}(N/d) +    \frac{\phi(k/d)  }{(k/d) }\log (N/k)
\Big)
\cr
   &\le & C   \frac{\phi(k/d)  }{(k/d) }\log (\frac{eN}{k})
 .
 \end{eqnarray}
 Henceforth,
\begin{eqnarray}\label{bba1}\sum_{k <\ell \le  N}\frac{(\ell, k)}{\ell }&\le& C    \log (\frac{eN}{k})
\sum_{d|k} \frac{\phi(k/d)  }{(k/d) }  =  C    \log (\frac{eN}{k})
\sum_{\k|k} \frac{\phi(\k)  }{\k }  .
\end{eqnarray}
  Now, similarly by writing $\ell= \l d$, $k=\k d$, $(\ell, k)=d$, we get
$$\sum_{1\le \ell <k}\frac{(\ell, k)}{\ell \vee k}= \sum_{d|k }\frac{1}{  (k/d)} \sum_{\frac{1}{d}\le \l <k/d\atop (\l,\k)=1} 1
\le\sum_{d |k}\frac{\phi(k/d)}{  k/d}=\sum_{\k |k}\frac{\phi(\k)}{  \k}   .
$$
 Consequently,
 $$\sum_{1\le \ell \le  N\atop \ell\not= k}\frac{(\ell, k)}{\ell \vee k} \le C\log (\frac{eN}{k})  \sum_{\k |k}\frac{\phi(\k)}{  \k} . $$
The proof is now complete.
 \end{proof}

  We pass to mean estimates.  Lemma \ref{quadraf} implies
$$   \Big|  \sum_{i,j=1 }^n  x_ix_j\a_{i,j}-\sum_{i =1 }^n  x_i^2\a_{i,i} \Big|\le  \frac1{2}\sum_{ i=1}^n x_i^2\Big(  \sum_{\ell
=1\atop
\ell\not = i }^n (|\a_{i,\ell} | +|\a_{
\ell ,i} |) \Big)  ,
$$
which is   extremely useful.
 Another  simple consequence   concerns Riesz sequences.
\begin{definition}   \label{rieszb}
 A sequence of vectors $\{v_i ,i\ge 1\}$  in a Hilbert space
$ H$ is called a   Riesz sequence  if there exist positive constants  $C_1$, $C_2$  such that
$$ C_1 \Big( \sum_{i=1}^n \vert x_i\vert^2 \Big) \le
\Big\| \sum_{i=1}^n x_i v_i \Big\|^2 \le  C_2 \big( \sum_{i=1}^n
\vert x_i\vert^2 \big)  ,
$$ for all sequences of scalars
$\{x_i ,1\le i\le n\}$.
 \end{definition}
\begin{theorem}\label{rs} Let ${\bf v}=\{v_i ,i\ge 1\}$ be  a sequence of vectors    in a Hilbert space $ H$ such that
\begin{equation}\label{rieszbcond}   \sup_{i\ge 1}\sum_{  j\not = i }   |\langle v_i, v_j\rangle | < \inf_{i\ge 1}\|v_i\|^2.
\end{equation}
 Then $\{v_i ,i\ge 1\}$ is a Riesz sequence. \end{theorem}
\begin{proof} Put
$$b({\bf v})= \sup_{i\ge 1}\sum_{j \ge 1
 \atop j\not = i }   |\langle v_i, v_j\rangle |.$$
   By taking
$\a_{i,j}= \langle v_i, v_j\rangle$ in Lemma \ref{quadraf}, we get
$$   \Big|  \Big\| \sum_{i=1}^n x_i v_i \Big\|^2-\sum_{i=1}^n  x_i^2\|v_i\|^2 \Big|\le  \frac1{2}\sum_{i=1}^n  x_i^2\Big( \sum_{\ell=1\atop
\ell\not = i}^n
   (|\a_{i,\ell} | +|\a_{
\ell ,i} |) \Big)\le  b({\bf v}) \sum_{i=1}^n x_i^2 .
$$
Hence,
 $$  \Big(   \inf_{i\ge 1}\|v_i\|^2-  b({\bf v}) \Big)\sum_{i=1}^n  x_i^2\le  \Big\|\sum_{i=1}^n x_i v_i \Big\|^2\le \Big( \sup_{i\ge 1}\|v_i\|^2  +
 b({\bf v}) \Big)\sum_{i=1}^n   x_i^2   .$$
 \end{proof}

 Hedenmalm, Lindquist and Seip \cite{HLS1}, \cite{HLS2} proved   that if
   $g(t)\sim \sum_{k=1}^\infty \p_k\cos 2\pi kt$, $g\in L^2(\T)$,
then $\{g_n, n\ge 1\}$ (recall that $g_n(x)=g(nx)$) is a Riesz sequence in $L^2(\T)$ if and
only if  the Dirichlet series
$\  \sum_{n=1}^\infty \p_nn^{-s}$\ is analytic and bounded away from 0 and
$\infty$ in the whole right half-plane  $\Re z>0$, i.e.
\begin{equation}\label{hlsdir}
\d\le \Big|\sum_{n=1}^\infty \p_nn^{-\s-it}\Big|\le \D,\qq \hbox{for} \ \s>0,
\end{equation}
with some positive constants $\d$ and $\D$.
\vskip 3pt
 In view of Theorem \ref{rs}, we deduce that a sufficient condition for (\ref{hlsdir}) to be satisfied is
\begin{equation}\label{rieszbcond1}
\sup_{i\ge 1}\sum_{j \ge 1 \atop j\not = i }
  |\langle g_i, g_j\rangle | <  \|g\|^2.  \end{equation}

Concerning the class of examples considered in the Introduction, we deduce
\begin{corollary} \label{corrb} Let $f$ be defined as in (\ref{funct}) with $1/2<s\le 1$. Let $\{n_i, i\ge 1\}$ be increasing and
satisfying
\begin{eqnarray} \label{sp2} \sup_{i\ge 1}\sum_{j \ge 1 \atop j\not = i }
  \frac{(n_i,n_j)^{2s}}{n_i^sn_j^s} < 1.
\end{eqnarray}
Then $\{f_{n_i}, i\ge 1\}$   is a Riesz sequence in $L^2(\T)$.
\end{corollary}
\begin{remark} Br\'emont (\cite{Br} Theorem 1.2-i)) showed, using    M\"obius orthogonalization,  that the  
sequence $\{f_{n_k}, k\ge 1\}$   is a Riesz sequence in $L^2(\T)$ whenever 
$$n_{k+1}/n_k\ge c >1. $$
  If $c>3$, this follows immediately from Corollary \ref{corrb} 
 since\begin{eqnarray*} \sum_{j \ge 1 \atop j\not = i } 
  \frac{(n_i,n_j)^{2s}}{n_i^sn_j^s}=2  \sum_{j > i } 
  \frac{(n_i,n_j)^{2s}}{n_i^sn_j^s}\le 2\sum_{j > i }
  \big(\frac{n_i }{ n_j }\big)^s \le 2\sum_{j > i } c^{-(j-i)}=\frac{2}{c-1}<1.
 \end{eqnarray*}
For $c>1$, this is however a special case of Kac's result \cite{Ka}  later extended by Gaposhkin, since  the square modulus of continuity of $f$
$$
\o_2(\d, f)=\sup_{0<h \le \delta} \Big\{\int_0^1 |f(x+h)-f(x)|^2 \dd x
\Big\}^{1/2}
$$
  satisfies
 $
\o_2(\d, f) =\mathcal{ O}  (  \delta ^{ \e} )
 $ for some $\e>0$. Let indeed $f(x) = \sum_{m=1}^\infty \frac{\sin 2\pi mx}{m^s} 
$, 
 where $s>1/2$. Using formula (3.2) in \cite{Z} p.241, gives
$$\int_0^1 |f(x+h)-f(x)|^2 \dd x= C\sum_{m=1}^\infty \frac{\sin^2(\pi mh)}{m^{2s}}\le C\sum_{m=1}^\infty \frac{ ( mh\wedge 1) }{m^{2s}}
\le C h^{2s-1}.$$\end{remark}
 \vskip 6pt We now  give   the
 \begin{proof} [Proof of Lemma \ref{basic0}]   Putting
    $Z_j= \sum_{k\in K} c_k  e_{jk}$, we have   \begin{eqnarray}\label{basic}  \big \|\sum_{  j\in   J}a_j \sum_{k\in K} c_k  e_{jk}\big\|_2^2&=& \big \|\sum_{  j\in
J}a_jZ_j\big
\|_2^2   =    \sum_{  j\in   J}a^2_j\big \|Z_j\big \|_2^2
+  \sum_{  i\not= j\atop
i,j \in   J}a_ia_j\langle Z_i , Z_j \rangle
\cr &=&\sum_{  j\in   J}\sum_{k\in K} a^2_j c_k^2
 + \sum_{k, \ell\in K}c_kc_\ell \sum_{  i\not= j\atop
i,j \in   J}a_ia_j  {\bf 1}_{\{jk=i\ell\}} .
  \end{eqnarray}

Let $a=J_-,b=J_+$. Fix $k, \ell\in K$. The equation $jk=i\ell$,   $i\not=j$, $i,j \in   J$,   being impossible for
$k=\ell$, let
     $k<\ell$.  Writing     $d=(k, \ell)$, $k=k'd$,
$\ell=\ell'd$, the equation becomes $jn'_k=i \ell'$. General solutions are $j=u \ell'$, $i=uk'$.  Then
$$a\le  j\le b \qq \Rightarrow\qq  \frac{ad}{\ell }=\frac{a}{\ell'} \le u  \le   \frac{b}{\ell'}=\frac{bd}{\ell } .$$
Operating similarly for $i$, it follows that
 $$\frac{ (k, \ell)}{  k } a\le u  \le   \frac{ (k, \ell)}{\ell  } b.$$
 Thus  solutions exist only if $ \ell$ and  $k$ are such that
$$\frac{ \ell }{  k } \le  \frac{b}{a}. $$
  And in that case, their number is   bounded by
$$ (k,
\ell)\Big(\frac{ b }{\ell  }  -\frac{ a}{  k }  \Big).$$
 Thus
  $$\Big|\sum_{  i\not= j\atop
i,j \in   J}a_ia_j    {\bf 1}_{\{jk=i\ell\}}\Big|\le   \sup_{j\in J}  a_j^2(k,
\ell)\Big(\frac{ b }{\ell  }  -\frac{ a}{  k }  \Big)\le  (b-a) \sup_{j\in J}  a_j^2 \frac{ (k,
\ell)}{\ell  }        . $$
 But this bound remains trivially valid if $\frac{ \ell }{  k } >  \frac{b}{a}$, since the sum in the left term is empty.   The case
$\ell<k$ being identical, it follows that
$$\Big|\sum_{  i\not= j\atop i,j \in   J}a_ia_j    {\bf 1}_{\{jk=i\ell\}}\Big|\le   \sup_{j\in J}  a_j^2\Big(\frac{ (k,
\ell)}{\ell\vee k } b-\frac{ (k, \ell)}{\ell\wedge k } a \Big)\le  (b-a) \sup_{j\in J}  a_j^2 \frac{ (k,\ell)}{\ell\vee k }        . $$
 By reporting in (\ref{basic}), next using Lemma \ref{quadraf}, we get
\begin{eqnarray}\label{basicc}  \Big| \big \|\sum_{  j\in   J}a_j \sum_{k\in K} c_k  e_{jk}\big\|_2^2 - \sum_{  j\in   J}\sum_{k\in K} a^2_j c_k^2\Big| &\le &
  (b-a) \sup_{j\in J}  a_j^2  \sum_{k, \ell\in K}|c_k||c_\ell| \frac{ (k,\ell)}{\ell\vee k }
  \cr &\le &
  (b-a) \sup_{j\in J}  a_j^2  \sum_{k \in K} c_k^2\max(1,\t_K(k)) .
  \end{eqnarray}
   \end{proof}

 By combining Lemma \ref{gb}  with estimate b) of Lemma \ref{basic0}, we immediately get
\begin{corollary}\label{bound(d)} Under assumptions of Lemma \ref{basic0},\begin{eqnarray*} \Big| \big \|\sum_{k\in K}
c_k\big(\sum_{  j\in   J}a_j    e_{jk}\big)\big\|_2^2-
\sum_{  j\in   J}a^2_j\sum_{ k\in K}  c_k^2\Big|\le  C|J|
\big( \sup_{   j \in J} a_j^2  \big)\sum_{ k\in K}  c_k^2\log (\frac{eK_+}{k})A(k) .
\end{eqnarray*}
\end{corollary}

\begin{remark}  The factor $\log (\frac{eK_+}{k})$ appearing in Lemma \ref{basic01} and in   Corollary \ref{bound(d)} is very
restrictive, but seems unavoidable. However, when the coefficients $c_k$, $k\in K$ are commensurable, it can be removed.  We indeed also have,
\begin{eqnarray*}\Big| \big \|\sum_{k\in K} c_k\big(\sum_{  j\in   J}a_j    e_{jk}\big)\big\|_2^2-
 \sum_{  j\in   J}a^2_j\sum_{ k\in K}
 c_k^2\Big|\le
\sup_{k\in K}  c_k^2\Big(   \sum_{  k\in K } A(k)\Big)|J|\sup_{j\in J}  a_j^2  .
\end{eqnarray*}
We omit the proof.\end{remark}
 \vskip 5pt
We pass to    orthogonality results. Let $M\ge \m>1$. Let
$K,L ,  I,J$  be    sets of positive integers such that:
   {\it For some  integers $B,u,v\ge 0$ with $|v-u|>1$},
 \begin{equation}\label{mub}K\cup L\subset [\m^B
\m^{B+1}[ \qq{\rm and}\qq I \subset [M^u , M^{u+1}[, \    J\subset [M^v , M^{v+1}[ .
\end{equation}
  Put
$$ T_H(G) =\sum_{k\in H} c_k\sum_{j\in G }   a_j e_{kj}, \qq H\in \{K,L\},\ G\in \{I,J\}.$$
 \begin{lemma} \label{ortholem}Under assumption (\ref{mub}),
 $\langle T_K(J),T_L(I)\rangle = 0$.
\end{lemma}
\begin{proof}First notice that for any  $k\in K,\ell \in L  $, the ratio $\ell/k$ satisfies $1/\m<\ell/k<\m$. Now plainly,
$$\langle T_K(J),T_L(I)\rangle=\sum_{k\in K\atop  \ell \in L} c_kc_\ell \sum_{| i|\in I \atop | j|\in J }   a_ja_i
\d_{jk=i\ell}
  . $$
Suppose $v>u+1$.  Then $\frac{|j|}{|i|}\ge M^{v-(u+1)}\ge M$. The equation $jk=i\ell$ is impossible. Indeed,
$$\frac{j}{i}=\frac{\ell}{k}\qq \Rightarrow \qq M\le \frac{\ell}{k}<\m.$$
Hence a contradiction since we assumed $M\ge \m$. If $u>v+1$, then $\frac{|i|}{|j|}\ge M^{u-(v+1)}\ge M$, and we arrive
similarly to $M\le \frac{k}{\ell}<\m$. \end{proof}

  Put for any finite set $K $ of integers,
$$ T_K(v):=T_K([M^v
 ,M^{v+1}[)       .$$
\begin{corollary}\label{sumKu}\begin{eqnarray*} \big\|\sum_{k\in K} c_k f_k\big\|_2^2\le 3\sum_{u=0}^\infty \big\|
T_K(u)\big\|_2^2 .
\end{eqnarray*}
When further the coefficients $\ua$, $\uc$ have each constant signs, we also have
\begin{eqnarray*}   \sum_{u=0}^\infty \big\|
T_K(u)\big\|_2^2 \le \big\|\sum_{k\in K} c_k f_k\big\|_2^2\le 3\sum_{u=0}^\infty \big\|
T_K(u)\big\|_2^2 .
\end{eqnarray*}  \end{corollary}
\begin{proof} Set  $\D (v)  =\sum_{M^v
\le| j|<M^{v+1} } a_j e_j$, $v\ge 0$.
As $f= \sum_{u=0}^\infty \D(u)$,  Lemma \ref{ortholem}  implies
\begin{eqnarray}\label{sumK}\big\|\sum_{k\in K} c_k f_k\big\|_2^2&=& \Big\|\sum_{u=0}^\infty \sum_{k\in K} c_k\D_k(u)\Big\|_2^2=
\Big\|\sum_{u=0}^\infty T_K(u)\Big\|_2^2
\cr &= &\sum_{u=0}^\infty \big\|  T_K(u)\big\|_2^2+2\sum_{u=0}^\infty\langle T_K(u),T_K(u+1)\rangle,
\end{eqnarray}
which easily  allows  to conclude. \end{proof}


 Now recall   Schur's Theorem (\cite{O}, p.\ 56).
\begin{lemma}\label{schur} Let $X$ be a bounded interval of the real line endowed with the normalized Lebesgue measure. Let $\{f_k, 1\le
k\le n\}$ be measurable functions on a measurable set $E\subset  X$, $\l(X\backslash E)>0$. These functions   can be extended   to an
orthonormal system on $X$ if and only if the following condition is satisfied
\begin{equation}\label{qos} \Big\|\sum_{k=1}^n  c_kf_k \Big\|_2^2\le \sum_{k=1}^nc_k^2\qq \qq (\forall  c_1,\ldots, c_n).
\end{equation}
\end{lemma}
It is true by induction for infinite sequences. The main argument
of the proof is   that
$I-G$, where
$G$ is the Gram matrix of the system  i.e. $G=(\g_{k,\ell}  )$, $\g_{k,\ell}=\int _Ef_k  f_\ell \dd x$, is nonnegative definite. Hence,
it is possible to construct on  $E^c$ a system of functions having $I-G$ as Gram matrix.
\begin{remark} \label{wschur} More generally, if  for positive reals    $\{\d_k, 1\le
k\le n\}$ we have
\begin{equation}\label{qosw}  \Big\|\sum_{k=1}^n  c_kf_k \Big\|_2^2\le  \sum_{k=1}^n\d_k\,c_k^2\qq \qq (\forall  c_1,\ldots, c_n),\end{equation}
then $\{f_k, 1\le k\le n\}$  can be extended   to an orthogonal
system $\{\xi_k, 1\le k\le n\}$ on $X$ satisfying $\|\xi_k\|_2=\sqrt
\d_k$ for all $k$.
     Therefore   the series
$\sum_k  {c_k}  f_k $ converges almost everywhere for all sequences $\{c_k , k\ge 1\}$ such that $\{c_k\sqrt \d_k, k\ge 1\}$ is universal.
\end{remark}

\vskip 2pt


A complete characterization of universal coefficient sequences  has been recently
obtained  in \cite{P} by Paszkiewicz, solving a long standing open problem.  Let
$$A^\infty= \Big\{\sum_{k\ge m} c_k^2; m=1,2,\ldots\Big\} $$
  \begin{theorem} A sequence  of coefficients $\{c_k, k\ge 1\}$, $\sum_{k } c_k^2\le 1$ is universal if and only if there exists a
finite measure $m$ on
$A^\infty$ such that
\begin{equation} \label{mmo}\sup_{t\in A^\infty}\int_0^1 \frac{\dd \e }{\sqrt{m(  (t-\e^2, t+\e^2))}} <\infty.
\end{equation}
\end{theorem}
  A measure $m$ such that (\ref{mmo}) holds is   called a {\it majorizing measure}.   Paszkiewicz showed with
this deep   result the great success of the majorizing measure approach, a technic which has been considerably developed over the years by
Talagrand, more recently by Bednorz,
 and also applied by the
 second named author to some   convergence problems in analysis. Paszkiewicz  further obtained other characterizations
involving   convolution powers of some natural operator.


 \section{\bf Proof of Theorem \ref{square}}

\medskip
 Let $\mathcal N=\{n_k, k\ge 1\}$ be an increasing sequence of positive integers. Let $\m>1$
and consider the "trace" of
$\mathcal N$ over the geometric partition of $\N$ associated to the sequence $\{\m^j, j\ge 0\}$, namely the sets
$$ N_j =\mathcal N\cap [\m^j, \m^{j+1}[, \qq j=0,1,\ldots $$
Some of them may be empty, so let $\{N^*_j, j\ge 0\}$ denote the subsequence obtained after having removed all empty sets.  By assumption,
 $$
\sup_{j\ge 0}\t_{\mathcal N\cap [\m^j, \m^{j+1}[}<\infty  .
$$
 Let  $K\subset N^*_j$ for some $ j$. Let $M>\m$.  Applying Lemma \ref{basic0} to
$$ T_K(v)   =\sum_{k\in K} c_k\sum_{M^v \le j\le M^{v+1} }   a_j e_{kj},$$
gives
 \begin{eqnarray}\label{inter}& &\Big| \big \|T_K(v)\big\|_2^2-  \sum_{M^v \le j\le M^{v+1} } a^2_j\sum_{ k\in K}
c_k^2\Big|\le  C \t_K M^{v+1}
\big( \sup_{M^v \le j\le M^{v+1} }  a_j^2  \big)\big(\sum_{ k\in K}  c_k^2  \big).
\end{eqnarray}
 Using  Corollary \ref{sumKu}, we can bound as follows
\begin{eqnarray}\label{sumK0}\big\|\sum_{k\in K} c_k f_k\big\|_2^2&\le & 3\sum_{v=0}^\infty \big\|
T_K(v)\big\|_2^2  \cr &\le &C (1+\t_K ) \big(\sum_{ k\in K}  c_k^2  \big)\sum_{v=0}^\infty
M^{v+1}
\big( \sup_{M^v \le j\le M^{v+1} }  a_j^2\big)
\cr &\le  &C_{\mathcal N, \m, f}   \big(\sum_{ k\in K}  c_k^2  \big) .
 \end{eqnarray}
  By taking $K= N^*_j$ and summing over   $ j$, we get
$$\sum_{j\ge 0}\big\|\sum_{k\in N^*_j} c_k f_k\big\|_2^2 \le C_{\mathcal N, \m}\|f\|_2^2 \sum_{j\ge 0}\sum_{k\in N^*_j} c_k^2\le
C_{\mathcal N, \m}\|f\|_2^2 \sum_{k\ge 0}  c_k^2, $$
as claimed.
Now, in the case the coefficient sequences have each constant signs, we appeal to the second part of Corollary \ref{sumKu} and use the
fact that $ \t_{\mathcal N\cap [\m^j, \m^{j+1}[} =o(1) $, by assumption (\ref{hyp2}) to conclude.

\bigskip
\section{\bf Proof of Theorem \ref{thN}}

\medskip
 We  decompose $f$ into a regular part and an irregular part,  $f=f^{\flat}+
f^{\sharp}$. Here $f^{\flat}=\sum_{\ell} a_\ell^{\flat}  e_\ell  $, $a_\ell^{\flat}=a_\ell
 \chi\{|a_\ell|>\e_\ell\}$ is the regular component of $f$ and will be directly controlled by means of Carleson-Hunt's theorem \cite{H}.
  For the control of  the irregular component
$f^{\sharp}$, arithmetical considerations are needed.
\vskip 2pt
  Plainly,
 \begin{eqnarray} \sup_{V\le u\le v\le W}\big|\sum_{u\le n\le v}c_nf^{\flat}_{  k_n}\big| &=& \sup_{V\le u\le v\le W}\big|\sum_{\ell}
a_\ell^{\flat}  \sum_{u\le n\le v}c_n  e_{\ell k_n} \big|
\cr &\le & \sup_{ V\le u\le v\le W}\sum_{\ell} |a_\ell^{\flat}|\big|\sum_{u\le n\le v}c_ne_{\ell k_n}\big|
\cr &\le &\sum_{\ell} |a_\ell^{\flat}|\sup_{ V\le u\le v\le W}\big|\sum_{u\le n\le v}c_ne_{\ell k_n}\big|.
\end{eqnarray}
   By using Carleson-Hunt's theorem \cite{H},
\begin{eqnarray}\label{ch}\Big\|\sup_{V\le u\le v\le W} \big|\sum_{n=u}^vc_nf^{\flat}_{  k_n}\big|\Big\|_2&\le &\sum_{\ell}
|a_\ell^{\flat}|\Big\|\sup_{ V\le u\le v\le
W}\big|\sum_{u\le n\le v}c_ne_{\ell k_n}\big|\Big\|_2
 \cr &\le & A\sup_{\ell}\Big\|\sup_{ V\le u\le v\le
W}\big|\sum_{u\le n\le v}c_ne_{\ell k_n}\big|\Big\|_2\le CA  \Big(\sum_{k=V}^Wc_k^2\Big)^{1/2}.
\end{eqnarray}
Therefore,    the     sequence $\big\{\sum_{n= 1}^N c_nf^{\flat}_{  k_n}, N\ge 1\big\}$ has oscillation  near infinity tending to zero
a.e. In other words, the series  $ \sum_{n} c_nf^{\flat}_{  k_n} $ converges a.e.
\vskip 3pt
To control the sums related to the irregular component  we need an extra lemma.
\begin{lemma} \label{basic0101}  For   any  finite set
$K$,
\begin{eqnarray*}\Big\|\sum_{k\in K}c_kf^\sharp_k \Big\|_2 \le   CB\Big(   \sum_{ k\in K}  c_k^2  \t_K(k)\Big)^{1/2}   ,
  \end{eqnarray*}
where $B$ is defined in assumption (\ref{condaj1}).
If $j_s=M^s$ for some $M>1$,  let
 $B_1=\sum_{s} M^{s+1}  \e^2_{Ms}$. Then,
\begin{eqnarray*}\Big\|\sum_{k\in K}c_kf^\sharp_k \Big\|_2 \le   CB_1^{1/2}\Big(   \sum_{ k\in K}  c_k^2  \t_K(k)\Big)^{1/2}   ,
  \end{eqnarray*}
\end{lemma}
 \begin{proof}
  Let $J_s= [j_s, j_{s+1}[$.   By Lemma \ref{basic0},
 \begin{eqnarray*}\|\sum_{k\in
K}c_k\sum_{j\in J_s }a_j^{\sharp} e_{kj}\|_2^2& \le &\sum_{  j\in   J_s }\sum_{k\in K} {a_j^{\sharp}}^2 c_k^2+ C(j_{s+1}-j_{s })
\e_{j_s}^2  \sum_{ k\in K}  c_k^2  \t_K(k)
  \cr &\le &     C(j_{s+1}-j_{s })
\e_{j_s}^2\Big(   \sum_{ k\in K}
c_k^2 \max(1, \t_K(k))\Big).
\end{eqnarray*}
 Thus
\begin{eqnarray*}\big\|\sum_{k\in K}c_kf^\sharp_k \big\|_2 &\le&   \sum_s \big\|\sum_{k\in K}c_k\sum_{j\in J_s}a_j^\sharp e_{kj}
\big\|_2   \le   C\Big(\sum_{s} j_{s+1}^{1/2} \e_{j_s}\Big) \Big(   \sum_{ k\in K}
c_k^2  \t_K(k)\Big)^{1/2}
\cr &  = & CB\Big(   \sum_{ k\in K}  c_k^2  \t_K(k)\Big)^{1/2}
\end{eqnarray*}
Further, when $j_s=M^s$ for some $M>1$,
 by Corollary \ref{sumKu}, next Lemma \ref{basic0},
\begin{eqnarray}\big\|\sum_{k\in K}c_kf^\sharp_k \big\|_2^2 &\le&   C\sum_s \big\|\sum_{k\in K}c_k\sum_{j\in J_s}a_j^\sharp e_{kj}
\big\|_2^2   \le   C\Big(\sum_{s} j_{s+1}  \e^2_{j_s}\Big) \Big(   \sum_{ k\in K}
c_k^2  \t_K(k)\Big)
\cr &  = & CB_1\Big(   \sum_{ k\in K}  c_k^2  \t_K(k)\Big) .
\end{eqnarray}
  \end{proof}

Now we can finish the proof of Theorem \ref{thN}.
   By   Remark \ref{wschur},
 the series
 $\sum_{n\ge 1}
c_{ n} {f^\sharp_{k_n}} $,  converges a.e. for any coefficient sequence $\{c_n, n\ge 1\}$ such that $
 \{c_n {\sqrt{\t_K({k_n})}}, n\ge 1\}$ is universal. And this follows from assumption (\ref{cconv1}).
 Since we have seen that the series $ \sum_{n}   c_n  f^\flat_{k_n} $ converges a.e., we deduce that the series $ \sum_{n}   c_n
f_{k_n}
$ converges a.e.

\bigskip
\section{\bf Proof of Corollary \ref{cNc}}

\medskip
   By Rademacher-Menshov's Theorem, the sequence $\{c_n {\sqrt{\t_K({k_n})}}, n\ge 1\}$ is universal
if
 $\sum_n c_n^2\t_K({k_n} )(\log n)^2<\infty$.
But by Lemma \ref{gb},
\begin{eqnarray*}\sum_v\sum_{2^v<n\le 2^{v+1}} c_n^2\t_K({k_n} )(\log n)^2&\le& C\sum_v\sum_{2^v<n\le 2^{v+1}} c_n^2A_K({k_n} )(\log
n)^2
\cr & = & C\sum_n c_n^2A_K({k_n} )(\log n)^2<\infty,
\end{eqnarray*}
by assumption. Hence, by    Theorem \ref{thN}  the series  $ \sum_{n}   c_n  f_{k_n} $ converges a.e. Taking in particular $K=\N$,
yields that $ \sum_{n}   c_n  f_{ n}
$ converges a.e., whenever
 $\sum_n c_n^2A (n )(\log n)^2<\infty$.
 \vskip 5pt
\begin{remark} Let $\{k_n, n\ge 1\}$ be an arbitrary increasing sequence of integers. The part of the proof concerning $f^\flat$ also
implies  that the series $\sum_{n} c_n f_{k_n}$ converges a.e. whenever $f\in
\hbox{Lip}_\alpha(\T)$ with $\a>1/2$, and for any coefficient sequence such that $\sum_n c_n^2<\infty$, which much improves upon
Corollaries 2.3, 2.3*, 2.5*, 2.6 in \cite{BW}.
\end{remark}
\section{\bf Proof of Theorem \ref{cNc1}}
We produce it   for $k_n=n$, the case of an arbitrary increasing sequence $k_n$ being  treated identically. By specializing (\ref{61}) for
$K=[2^r, 2^{r+1}]$, we get   
  \begin{eqnarray}\label{610}\qq    \big\|\sum_{2^r\le k< 2^{r+1}} c_k f_k\big\|_2^2  \le
  C_s  \sum_{2^r\le k< 2^{r+1}} \s_{1-2s}(k) c_k^2
  ,  \end{eqnarray}
when $1/2< s\le 1$. 
Thus by assumption (\ref{cconv2})
\begin{eqnarray*}
 \sum_{r=1}^\infty \int_0^1 r^2\Big| \sum_{j=2^r+1}^{2^{r+1}}c_j
f_j(x)\Big|^2 dx&\le& C_s\, 
\sum_{r=1}^\infty r^2 \sum_{k=2^r+1}^{2^{r+1} }\s_{1-2s}(k) c_k^2\cr
& \le & C_s\,\sum_{r=1}^\infty \sum_{j=2^r+1}^{2^{r+1}} \s_{1-2s}(k) c_k^2 (\log k)^2
 <\infty.
\end{eqnarray*}
Therefore
$$
\sum_{r=1}^\infty r^2\Big| \sum_{j=2^r+1}^{2^{r+1}}c_j
f_j(x)\Big|^2< \infty \qquad \text{a.e.}
$$
And the Cauchy-Schwarz inequality yields for any $1\le M<N$
\begin{eqnarray*}
  \Big|\sum_{j=2^M+1}^{2^N} c_j f_j(x)\Big|^2 &\le &\bigg(
\sum_{k=M}^{N-1} \Big|\sum_{j=2^k+1}^{2^{k+1}} c_j
f_j(x)\Big|\bigg)^2 \cr
&\le &\bigg(\sum_{k=M}^{N-1} \frac{1}{k^2}\bigg) \bigg(
\sum_{k=M}^{N-1} k^{2} \Big|\sum_{j=2^k+1}^{2^{k+1}} c_j
f_j(x)\Big|^2 \bigg) \cr
& \le &2\sum_{k=M}^\infty k^{2} \Big|\sum_{j=2^k+1}^{2^{k+1}} c_j
f_j(x)\Big|^2 \to 0,
\end{eqnarray*}
as $M\to\infty$. This implies that $\sum_{j=1}^{2^m} c_j f_j(x)$ 
converges a.e.\ as $m\to\infty$. Now by using again (\ref{61}) and
standard
maximal inequalities (see e.g. \cite{W}, Lemma 8.3.4) 
 we get
\begin{eqnarray*}
 \sum_{k=1}^\infty \bigg\| \max_{2^k+1\le i\le j\le 2^{k+1}}
|\sum_{\ell=i}^j c_\ell f_\ell |\bigg\|^2  &\le& C_s\,  \sum_{k=1}^\infty
k^2 \bigg(\sum_{\ell=2^k+1}^{2^{k+1}} \s_{1-2s}(\ell) c_\ell^2 \bigg)\cr
 &\le& C_s\, 
\sum_{\ell=1}^\infty  \s_{1-2s}(\ell) c_\ell^2 (\log \ell)^2 <\infty, 
\end{eqnarray*}
which implies
\begin{equation}\label{e}
\max_{2^k+1\le i\le j\le 2^{k+1}} \Big|\sum_{\ell=i}^j c_\ell
f_\ell( x)\Big|\to 0 \qquad \text{a.e.}
\end{equation}
completing the proof of the theorem.

\bigskip
 {\bf Acknowledgment.} We wish to thank   Pennti Haukkanen for  useful references, notably the
recent article of  Julien Br\'emont, which we discovered while   this work was much advanced.

 {}


\baselineskip 9pt \begin{thebibliography}{99}
\bibitem{A} Aisleitner C. (2011) {\it Convergence of $\sum  c_kf(kx) $ and the Lip${}_\a$ class},  to appear.
 \bibitem{B}  Bellman R.  {\sl Introduction to Matrix Analysis}, Sd Ed.,    Classics in Appl. Math. {\bf 19}  Siam, Philadelphia
(1997).
\bibitem{Be} Berkes I. {\sl On the convergence of $\sum_n c_nf(nx)$ and the Lip 1/2 class}, Trans. Amer.
Math. Soc.  {\bf 349} no10,  4143-4158.
 \bibitem{BW}  Berkes I., Weber M. (2009) {\sl On the convergence of $\sum c_k f(n_k x)$},    Memoirs  of the  A.M.S. {\bf 201} no. {\bf
943},
   vi+72p.
 \bibitem{Br}  Br\'emont J. (2011) {\sl Davenport series and almost sure convergence},  Quart. J. Math. {\bf 62},
825--843.
\bibitem{C} {Carleson, L.} (1966) {\sl On convergence and growth of
partial sums of Fourier series}, Acta Math. {\bf 116}, 135--157.
  \bibitem{G} Gaposhkin V.F. (1968)
{\sl On convergence and divergence systems},  Mat. Zametki {\bf 4},  253-260.
\bibitem{g70} Gaposhkin V.F. (1970). The central limit theorem for certain weakly dependent sequences. (Russian)
Teor. Verojatnost. i Primenen. {\bf 15}, 666--684.

 
  \bibitem{Gr} Gronwall T.H.  (1912)
{\sl Some asymptotic expressions in the theory of numbers},  Trans. Am. Math. Soc. {\bf 8},  118--122.
    \bibitem{HSW} Haukkanen P., Wang J., Sillanp\"a\"a J. (1997), {\sl On Smith's determinant},    Linear algebra and its Appl. {\bf 258},
  251--269.
 \bibitem{HLS1} {Hedenmalm H., Lindqvist P., Seip K.} (1997) {\sl A Hilbert space of Dirichlet
series and systems of dilated functions in $L^2([0,1])$}, Duke Math.
J. {\bf 86}, 1-37.
\bibitem{HLS2} {Hedenmalm H., Lindqvist P., Seip K.} (1999) {\sl Addendum to "A Hilbert
space of Dirichlet series and systems of dilated functions in
$L^2([0,1])$"}, Duke Math. J. {\bf 99}, 175-178.
\bibitem{HL}  S. Hong, R. Loewy, (2004)  {\sl Asymptotic behavior of eigenvalues of greatest common divisor matrices},  Glasgow Math. J. {\bf 46}, 551--569.\bibitem{H} Hunt  R.  (1968)  {\sl On the convergence of Fourier series},  Orthogonal Expansions and their Continuous Analogues (Proc.
Conf., Edwardsville,  {\bf   III} 1967),  235--255, Southern Illinois Univ. Press, Carbondale.
\bibitem{J}   Jaffard S. (2004)
{\rm  On Davenport expansions}, \emph{Proc. of Symp. in Pure Math.}  {\bf 72.1}, 273--303.
\bibitem{Ka} Kac, M. (1943) {\sl  Convergence of certain gap series},   {\it Ann. of Math.} {\bf 44},
411--415.
\bibitem{LR}  van Lint J.H., Richert H.-E.  (1965)
{\sl  On primes in arithmetic progressions},  {Acta Arith.}  {\bf XI}, 209--216.
\bibitem{LS}  Lindqvist P., Seip K.  (1998)
{\sl  Note on some greatest common divisor matrices},  {Acta Arith.}  {\bf LXXXIV 2}, 149--154.
 \bibitem{O}  Olevskii A. M. (1975)    {\sl Fourier series with respect to
general orthogonal systems},  Ergebnisse der Mathematik und ihrer
Grenzgebiete, Band {\bf 86}.
\bibitem{P}  Paszkiewicz A. (2010) {\sl A complete characterization of coefficients of a.e. convergent orthogonal series and majorizing measures}, Invent.
  Math. {\bf 180}, 55--110.
\bibitem{Ta1}{Tandori K.} (1957)   {\sl Zur Divergenz der Orthogonal
Reihe}, Acta Sci. Szeged   {\bf 18},  57--130.
  \bibitem{Ta2}{Tandori K.} (1965)   {\sl Bemerkung zur Konvergenz der
 Orthogonal Reihen}, Acta Sci. Szeged {\bf 26},  249--251.
\bibitem{T1}  Tenenbaum G. (1985) {\sl Sur la concentration moyenne des diviseurs}, Comment.  Math. Helv. {\bf 60}, 411--428.
\bibitem{T2}  Tenenbaum G. (1990) {\sl Introduction  \`a
la  th\'eorie analytique et probabiliste des nombres},
Revue de l'Institut Elie
Cartan  {\bf 13}, D\'epartement de Math\'ematiques de
l'Universit\'e de Nancy I.
\bibitem{To}  T\'oth L. (2010) {\sl A survey of the gcd-sum functions}, J. Integers Sequences  {\bf 13},
Article 10.8.1.
 \bibitem{W1}  {Weber M.} (2011) {\sl On systems of dilated functions},    C. R. Acad. Sci. Paris, Sec. 1 {\bf 349}, 1261--1263.
  \bibitem{W}  {Weber M.} (2009)   {\sl Dynamical Systems and Processes}, European Mathematical
 Society Publishing House, IRMA Lectures
  in Mathematics
 and Theoretical Physics {\bf 14}, xiii+759p.
 \bibitem{Wi} Wintner A. (1944): {\sl Diophantine approximation and Hilbert space},  Amer.
Journal of  Math.   {\bf 66},  564-578.
  \bibitem{Z} {Zygmund. A} [2002]:  {\sl Trigonometric series}, Third Ed. Vol. {\bf  1}\&{\bf  2} combined,
Cambridge Math. Library, Cambridge Univ. Press.  \end{thebibliography}
\end{document}